\documentclass[psamsfonts]{amsart}
\usepackage[utf8]{inputenc}
\usepackage{amsfonts}
\usepackage{hyperref}
\hypersetup{
	pdftitle={Classical symmetric varieties and quiver representations},
	pdfsubject={Mathematics},
	pdfauthor={Yau Wing Li},
	pdfkeywords={Classical symmetric varieties, Quiver representations, Borel orbits, Gabriel's theorem, Corona graphs, arXiv, Preprint}
}
\usepackage{amsmath}
\usepackage{xcolor}
\usepackage{amsthm}
\usepackage{pdflscape}
\usepackage{mathrsfs}
\usepackage{bbold}
\usepackage{stmaryrd}
\usepackage[skip=5pt, indent=10pt]{parskip}
\usepackage{caption}
\usepackage{float}
\usepackage{dsfont}
\usepackage{tikz}
\usepackage{amssymb}

\usepackage[all]{xy}

\usepackage{color}
\usepackage[shortlabels]{enumitem}
\usepackage{hyperref}
\usepackage{pgf,tikz}
\usepackage{mathrsfs}
\usetikzlibrary{arrows}
\usepackage{tikz-cd}
\usepackage{graphicx}
\usepackage{filecontents}

\newtheorem{thm}{Theorem}[section]
\newtheorem{cor}[thm]{Corollary}
\newtheorem{prop}[thm]{Proposition}
\newtheorem{lem}[thm]{Lemma}

\theoremstyle{definition}
\newtheorem{defn}[thm]{Definition}

\newtheorem{exmp}[thm]{Example}

\theoremstyle{remark}
\newtheorem{rem}[thm]{Remark}

\makeatletter
\let\c@equation\c@thm
\makeatother
\numberwithin{equation}{section}

\newcommand{\N}{\mathbb{N}}

\newcommand{\bsl}{\backslash}

\newcommand{\gm}{\mathbb{G}_m}

\newcommand{\ra}{\rightarrow}

\newcommand{\inj}{\hookrightarrow}

\newcommand{\on}{\operatorname}

\DeclareMathOperator{\Hom}{Hom}

\DeclareMathOperator{\chara}{char}

\DeclareMathOperator{\RHom}{\textbf{R} \kern -2pt \Hom}
\DeclareMathOperator{\Rhom}{\textbf{R}\kern -2pt \hom}
\DeclareMathOperator{\RuHom}{\textbf{R} \underline{\kern -0pt \Hom}}

\newcommand\cO{\mathcal{O}}

\newcommand\cS{\mathcal{S}}

\newcommand\cV{\mathcal{V}}

\newcommand{\CC}{\mathbb{C}}

\newcommand\so{\mathfrak{so}}

\numberwithin{equation}{section}

\renewcommand{\and}{\quad \on{and} \quad}

\renewcommand{\j}[1]{\langle{#1}\rangle}

\newcommand\quash[1]{}

\newcommand{\ov}{\overline}

\renewcommand\c\circ

\renewcommand\a\alpha
\renewcommand\b\beta
\renewcommand\d\delta

\renewcommand\r\rho

\newcommand\DD{\mathbb{D}}

\title{Classical symmetric varieties and quiver representations}

\begin{document}

\author{Yau Wing Li}

\begin{abstract}
The parametrization of Borel orbits on classical symmetric varieties $G/K$ is traditionally described via the combinatorics of clans, developed by Matsuki, Oshima, Wyser, and Yamamoto. We provide a quiver-theoretic reproof of the indexing sets for $G = GL_{n}$ or $Sp_{2n}$ and $K$ a two-block Levi subgroup using Gabriel’s Theorem. This method reveals an unexpected independence from algebraic closure: while orbit counts for most symmetric varieties are sensitive to the base field, our parametrizations remain valid over arbitrary fields for $GL_n$, and in characteristic not 2 for $Sp_{2n}$. We further reformulate these sets via matchings in corona graphs to derive new enumerative results, including the ultra log-concavity of orbit counts.
\end{abstract}

\address{University of Melbourne, School of Mathematics and Statistics, Parkville, VIC 3010, Australia}
\email{yauwingli57@gmail.com}

\date{\today}
\maketitle
\tableofcontents

\section{Introduction}

Historically, symmetric varieties have played a central role in the representation theory of real reductive groups, as seen in the foundational works of Vogan and Lusztig--Vogan \cite{LV, V}. Recently, these varieties have gained further attention due to the relative Langlands program \cite{CN}. Beyond representation theory, symmetric varieties have emerged as a natural generalization for modern Schubert calculus. In this context, one studies the equivariant cohomology and $K$-theory of the closures of symmetric subgroup orbits in the flag variety. While classical Schubert calculus focuses on Borel orbits, this modern framework treats these $K$-orbits as the fundamental geometric objects of study, where $K$ is the fixed-point subgroup of an involution on $G$.

Recent developments in this field involve the construction of analogous Schubert and Grothendieck polynomials for these symmetric settings. Various authors, including Hamaker, Marberg, Pawlowski, Wyser, and Yong \cite{WY14, WY17, HMP18, MP20}, have explored these polynomials to provide explicit formulas and combinatorial models for the closures of $GL_p \times GL_q$, $O_n$, and $Sp_{2n}$ orbits on the flag variety.

One of the most basic questions about symmetric varieties is to parametrize their Borel orbits. For classical groups over $\CC$, the orbit structure of symmetric varieties has been studied extensively by Matsuki, Oshima, Yamamoto, and Wyser \cite{MO, Ya, WyPhD}, who introduced combinatorial parametrizations using clans satisfying specific constraints. Their arguments involve a significant amount of linear algebra. 

In this paper, we provide a shorter and more conceptual proof of these parametrizations via the theory of quiver representations, using a variant of Gabriel’s Theorem due to Lusztig. The key observation is that in type AIII, e.g. $(G,K) = (GL_{m+n}, GL_m \times GL_n)$, $B \bsl G/K$ can be interpreted as the moduli stack of representations of a specific quiver with fixed dimension vector.

From our perspective, it is more convenient to describe these orbits using matchings on corona graphs rather than clans for certain classical types.\footnote{For instance, Can and U\u{g}urlu \cite{CU} introduced the terminologies, symmetric clans and ssymmetric clans, for types BI and CII.} Nevertheless, the two descriptions are clearly equivalent.

Let us first recall the definition of the corona graphs. Let $K_n$ be the complete graph on $n$ vertices. Let $K_{n}^{(2)}$ be the graph with $n$ vertices such that every pair of vertices is joined by exactly two edges. 

\begin{defn}
    For $n \geq 1$, we define the \emph{$n$-corona graph} $C_n$ to be the corona of the complete graphs $K_n$ and $K_1$. That is, for each vertex of $K_n$, we attach a new vertex connected to it by a single edge.
\end{defn}
Similarly, we define the \emph{$n$-double corona graph} $C_n^{(2)}$ to be the corona of the complete graphs $K_n^{(2)}$ and $K_1$. Figure \ref{fig: C6C4} below illustrates the graphs $C_6$ and $C_4^{(2)}$.

\begin{figure}[h]
\centering

\begin{minipage}[b]{0.45\textwidth}
\centering
\begin{tikzpicture}
[every node/.style={circle, draw, inner sep=1pt}, scale=0.5]

\foreach \i [count=\j from 0] in {1,...,6} {
    \node (v\i) at ({120 + 360/6 * \j}:2) {$v_\i$};
}

\foreach \i in {1,...,6} {
    \node (w\i) [circle, draw, inner sep=0.4pt] at ({120 + 360/6 * (\i-1)}:4) {$w_\i$};
    \draw (v\i) -- (w\i);
}

\foreach \i in {1,...,6} {
    \foreach \j in {1,...,6} {
        \ifnum\i<\j
            \draw (v\i) -- (v\j);
        \fi
    }
}

\end{tikzpicture}
\caption*{(a) $C_6$}
\end{minipage}
\hfill
\begin{minipage}[b]{0.45\textwidth}
\centering
\begin{tikzpicture}
[
    every node/.style={
        circle, draw, font=\small, minimum size=5mm, inner sep=1pt, align=center
    },
    scale=0.9
]

\foreach \i/\angle in {1/135, 2/225, 3/315, 4/45} {
    \node (v\i) at (\angle:1.25) {$v_\i$};
}

\foreach \i/\angle in {1/135, 2/225, 3/315, 4/45} {
    \node (w\i) at (\angle:2.5) {$w_\i$};
    \draw (v\i) -- (w\i);
}

\foreach \i in {1,2,3,4} {
    \foreach \j in {1,2,3,4} {
        \ifnum\i<\j
            \draw[bend left=10] (v\i) to (v\j);
            \draw[bend right=10] (v\i) to (v\j);
        \fi
    }
}

\end{tikzpicture}
\caption*{(b) $C_4^{(2)}$}
\end{minipage}

\caption{(a) Corona graph $C_6$, (b) Double corona graph $C_4^{(2)}$}
\label{fig: C6C4}
\end{figure}

Let us describe the Borel orbits of the symmetric variety in certain types. 

\begin{thm} \label{t: main thm}
Let $k$ be a field. Let $X$ be the homogeneous space $G/K$.
\begin{enumerate}
    \item For $(G,K) = (GL_{m+n}, GL_m \times GL_n)$, there is a bijection between the set of $B(k)$-orbits in $X(k)$ and the set of $m$-matchings in the corona graph $C_{m+n}$.
    
    \item Assume that $\operatorname{char}(k) \ne 2$. For $(G,K) = (Sp_{2m+2n}, Sp_{2m} \times Sp_{2n})$, there is a bijection between the set of $B(k)$-orbits in $X(k)$ and the set of $m$-matchings in the double corona graph $C_{m+n}^{(2)}$.
\end{enumerate}
The cases above correspond to the symmetric pairs of type AIII and type CII, respectively.
\end{thm}

\begin{rem}
    Here are some remarks about the theorem.
    \begin{itemize}
        \item When $k=\CC$, the statements have been stated (without proof) in Matsuki--Oshima \cite{MO}. Yamamoto \cite{Ya} and Wyser \cite{WyPhD} proved the statements using explicit linear algebra arguments. 
        
        \item In both statements, we have $X(k) = G(k)/K(k)$. See Section~\ref{ss: symplectic} for details.

        \item A significant feature of Theorem~\ref{t: main thm} is the independence of the orbit parametrization from the base field $k$. For general symmetric varieties, the number of $B(k)$-orbits often depends on the arithmetic of the field. For instance, in the case of $G = SL_2$ and $K = \mathbb{G}_m$, both the number of double cosets $B(k) \backslash G(k) / K(k)$ and the number of orbits $B(k) \backslash X(k)$ depend on the structure of the square classes $k^\times / (k^\times)^2$. See Example \ref{ex: orthogonal_rationality_failure} for a more complicated example. 

        \item This methodology offers a streamlined parametrization via bijections with matchings in corona graphs. While the cases of type AIII and CII exhibit a unique field-independence, the extension of this framework to other symmetric pairs generally requires the base field to be algebraically closed to ensure the existence of certain representatives. As an application of this approach, we provide the analogous statement for the type BI case $(G,K) = (O_{2m+2n+1}, O_{2m} \times O_{2n+1})$ in Section~\ref{ss: orthogonal group}. Furthermore, we expect this framework to extend naturally to other classical symmetric pairs arising from involutions of $(GL_{m+n}, GL_m \times GL_n)$, including $(Sp_{2n}, GL_n)$, $(SO_{2n}, GL_n)$, and $(O_{2m+2n}, O_{2m} \times O_{2n})$ over algebraically closed fields.
    \end{itemize}
\end{rem}

As mentioned above, our proof of statement (1) relies on the quiver variety description and Gabriel’s Theorem. In fact, each $B(k)$-orbit contains a representative that is a binary matrix, i.e., a matrix with entries that are either zero or one. We obtain the results for type CII by considering the fixed points of an involution of $GL_{2m+2n}$.

A small advantage of using matchings of corona graphs over clans is that we can use general results about matchings of graphs. Let $a_{m+n,m}$ (resp.\ $c_{m+n,m}$) be the number of orbits in Theorem \ref{t: main thm} (1) (resp. (2)). By applying general results about matchings of graphs, we derive several interesting numerical properties on $a_{m+n,m}$ and $c_{m+n,m}$. For instance, we have the following proposition (Proposition \ref{p: ultra logconcave}).

\begin{prop}
    Fix an integer $p>0$. The sequence $(a_{p,m})$ is a symmetric unimodal sequence for $0 \le m \le p$. Moreover, we have the inequality $$a_{p,m}^2 \ge \left(1+\frac{1}{m}\right) \left(1 + \frac{1}{p-m}\right)a_{p,m-1}a_{p,m+1},$$ for $0 < m < p$, that is $(a_{p,m})$ is ultra log-concave. The analogous statements hold for $c_{p,m}$.
\end{prop}

The paper is organized as follows. In Section 2, we establish the fundamental notations and conventions. Section 3 provides a brief review of the necessary results from the theory of quiver representations, focusing on Gabriel's Theorem. 

In Section 4, we employ the quiver-theoretic perspective to derive the matching parametrization for the symmetric variety of type AIII. Section 5 develops a more general group-theoretic framework using $\theta$-twisted involutions and the Bruhat decomposition. We show that the matchings in Section 4 are combinatorial manifestations of Weyl group involutions and provide field-independent proofs for type CII. Finally, in Section 6, we apply these combinatorial results to derive new  properties for the orbit counts.

\subsection{Acknowledgments} It is a pleasure to thank Dougal Davis, Pierre Deligne, Gurbir Dhillon, Mark Goresky, George Lusztig, Eric Marberg, Kari Vilonen, Daping Weng, Zhiwei Yun, and Xinwen Zhu, for inspiration and useful conversations. We also thank the anonymous referee for helpful comments and for pointing out several inaccuracies.

Y.W. Li was supported by the University of Melbourne,
IAS School of Mathematics, the National Science Foundation under Grant No. DMS-1926686, and the ARC
grant FL200100141.

\section{Notations}

Let $k$ be an arbitrary field and $\ov k$ be an algebraic closure of $k$. 

In the paper, $G$ always denotes a split reductive group over $k$ and $B$ is a Borel subgroup of $G$. Let $K$ be a subgroup of $G$ such that the number of $B$-orbits in $G/K$ is finite. When $\chara(k) \neq 2$, we can take $K$ to be the fixed point subgroup of an involution of $G$. 

For any algebraic group $H$ over $k$, we denote by $H(R)$ the $R$-points of $H$ for any $k$-algebra $R$. The following lemma about flag varieties is well-known.

\begin{lem} \label{l: flag variety rational}
    The group $G(k)$ acts transitively on $(G/B)(k)$. In other words, the $k$-points of the flag variety $G/B$ are $G(k)/B(k)$.
\end{lem}

Throughout this work, let $\mathfrak{so}_{2p}$ ($p \geq 4$) be the simple Lie algebra of Type $D_p$ split over $k$. We adopt the following convention for roots: the root system consists of $\pm L_i \pm L_j$ for $1 \leq i \neq j \leq p$, with the simple roots defined as $$ \a_i := L_i -L_{i+1}, \, 1 \leq i \leq p-1, \quad \a_p := L_{p-1} + L_{p}.$$

Our combinatorial framework utilizes the following quiver $Q_p$.

\begin{defn}
    For $p \geq 4$, let $Q_p$ be the unique quiver of the Dynkin diagram of type $D_{p}$ endowed with the orientation such that the trivalent vertex is the unique sink.  
\end{defn}
We refer to the two arrows originating at $\alpha_{p-1}$ and $\alpha_p$ as the \emph{branching arrows}.

Any positive root $\alpha$ in $\mathfrak{so}_{2p}$ can be interpreted as a dimension vector $d_{\alpha}$ on the quiver $Q_p$ and vice versa. We will identify $\alpha$ and $d_{\alpha}$ whenever there is no risk of confusion. Let $d_{m,n}$ be $L_1 + L_2 + \cdots + L_{m+n} + (n-m)L_{m+n+2}$ in the root lattice of $\so_{2(m+n+2)}$. See Figure \ref{fig: dim vector}.

\begin{figure}[H]
\begin{center}  
\begin{tikzpicture}[>=Stealth, scale=0.8, transform shape]

\node (1) at (0, 0) {1};
\node (2) at (2, 0) {2};
\node (3) at (4, 0) {3};
\node (dots) at (6, 0) {$\cdots$};
\node (p+q) at (8, 0) {$m+n$};
\node (p) at (10, 1) {$m$};
\node (q) at (10, -1) {$n$};

\draw[->] (1) -- (2);
\draw[->] (2) -- (3);
\draw[->] (3) -- (dots);
\draw[->] (dots) -- (p+q);
\draw[<-] (p+q) -- (p);
\draw[<-] (p+q) -- (q);

\end{tikzpicture}
\end{center}
\caption{Dimension vector $d_{m,n}$}
\label{fig: dim vector}
\end{figure}

Finally, for integers $i \leq j$, we let $[i, j]$ denote the set $\{i, i+1, \cdots, j\}$.

\section{Review on quiver representations}
In this section, we review some well-known results on quiver representations. Let $p= m+n+2$ throughout this section.  

Recall that a quiver representation $V$ of $Q_{p}$ (over $k)$ is equivalent to the following data:
\begin{enumerate}
    \item $k$-vector spaces $(U_i, U^m, U^n)$ for $ 1 \le i \le m+n$, 
    \item linear homomorphisms $\phi_i: U_i \ra U_{i+1}$ for $1 \le i < m+n$, and $\psi^j: U^j \ra U_{m+n}$ for $j=m,n$.
\end{enumerate}

Note that $\dim(V)$ is $d_{m,n}$ if and only if all of the dimensions of $U_i$ and $U^j$ coincide with their indices.

In order to relate the quiver representations with the Borel orbits of symmetric varieties in the next section, we introduce two subclasses of quiver representations. 

\begin{defn}
    A quiver representation of $Q_{p}$ is called \emph{injective} if the linear map associated to each arrow is injective.
\end{defn}

\begin{defn}
    A quiver representation of $Q_{p}$ is called \emph{complementary} if it is injective and the images of the two branching arrows form a direct sum decomposition of the target space.
\end{defn}

\begin{lem} \label{l: sum complementary}
    Let $V = V_1 \oplus V_2$ be a quiver representation of $Q_p$. Then $V$ is injective (resp.\ complementary) if and only if both $V_1$ and $V_2$ are injective (resp.\ complementary).
\end{lem}

\begin{proof}
    It follows from an elementary linear algebra argument. 
\end{proof}
    
The following theorem is a variant of Gabriel's Theorem  proved by Lusztig (cf. Propositions 4.12 and 4.15 in \cite{Lu}). Notice that $k$ is \emph{not} required to be algebraically closed in the theorem. 

\begin{thm}
     The map $$\dim: V \mapsto \dim(V)$$ induces a bijection between the set of indecomposable representations of $Q_p$ over $k$ and the set of positive roots of $\so_{2p}$.
\end{thm}

For a positive root $\alpha$, we denote by $V_{\alpha}$ the indecomposable representation associated to $\alpha$. The following lemma is proved by a direct computation and we omit the proof.

\begin{lem}\label{l: enumerate complementary}
   The representation $V_{\alpha}$ is complementary if and only if either $\alpha = L_i - L_p$, or $\alpha = L_i + L_j$ with $i \neq p-1$ and $j \neq p-1$.
\end{lem}

\begin{defn}
    A positive root $\alpha$ is called \emph{admissible} if the representation $V_{\alpha}$ is complementary.
\end{defn}


\begin{lem}\label{l: comp rep sum admissible}
    There is a natural bijection $\varphi$ from $\cV_{m,n}$ the equivalence classes of complementary representations $V$ with dimension $d_{m,n}$ and the ways to express $d_{m,n}$ as the sum of distinct admissible roots. 
\end{lem}

\begin{proof}
    We first construct the map $\varphi$. For a complementary representation $V$, we write $V = \bigoplus_{\alpha \in S} V_{\alpha}$ as the sum of indecomposable representations. Lemma \ref{l: sum complementary} implies these summands have to be complementary, i.e. $\alpha$ is admissible. We define $\varphi(V)$ to be $d_{m,n} = \sum_{\alpha \in S} \alpha$. The explicit description of admissible roots in Lemma \ref{l: enumerate complementary} shows that $\alpha$ are distinct. The map $\varphi$ is a bijection by the Krull-Schmidt Theorem.
\end{proof}

\section{Proofs of the result for $GL_{m+n}$}\label{s: GLmn}
Let $U$ be a $(m+n)$-dimensional $k$-vector space, and let $U = U^{m} \oplus U^{n}$ be a direct sum decomposition of $U$ into subspaces $U^m$ and $U^n$ of dimensions $m$ and $n$, respectively. The group $K:=GL(U^m) \times GL(U^n)$ is naturally a subgroup of $G:= GL(U)$. Let $X$ be the homogeneous space $G/K$. 

By choosing a basis $e_i$ for $U$ that is compatible with the direct sum decomposition, we may identify the pair $(G,K)$ as $(GL_{m+n}, GL_m \times GL_n)$. 

We begin with a lemma regarding $k$-points of $X$. 

\begin{lem} \label{l: rational point G K}
    The natural map $G(k)/K(k) \ra X(k)=(G/K)(k)$ is bijective.
\end{lem}

\begin{proof}
    From the theory of Galois cohomology \cite[Chapter I Proposition 36]{Serre}, we have the following exact sequence of pointed sets $$1 \ra K(k) \ra G(k) \ra X(k) \ra \on{H}^1(k, K).$$ Therefore, the natural map $G(k)/K(k) \ra X(k)$ is injective. By (generalized) Hilbert's Theorem 90, we know that $\on{H}^1(k, GL_r)$ is trivial for any $r$. K\"unneth formula implies that $\on{H}^1(k, K)$ is trivial. Hence, the natural map $G(k) \ra X(k)$ is surjective. 
\end{proof}

Combining with Lemma \ref{l: flag variety rational}, we conclude that the set of $B(k)$-orbits in $X(k)$ coincides with the set of $G(k)$-orbits in $(G/B)(k) \times X(k)$, which in turn corresponds to the set of double cosets $B(k) \bsl G(k)/ K(k)$. We will henceforth identify these sets without further mention.

\begin{prop}\label{p: bijection quiver and orbit}
    There is a bijection between the equivalence classes of complementary $k$-representations with dimension $d_{m,n}$ of $Q_{m+n+2}$ and the $B(k)$-orbits of $X(k)$. 
\end{prop}

\begin{proof}
    Note that the set of $B(k)$-orbits of $X(k)$ is in bijection with the set of $G(k)$-orbits of $(G/B)(k) \times X(k)$. We can identify $G(k)$ as the general linear group $GL(U)$ for a fixed $(m+n)$-dimensional $k$-vector space $U$. 
    
    By Lemma \ref{l: rational point G K}, we know that $X(k)$ parametrizes pairs of complementary $m$-dimensional and $n$-dimensional subspaces, $(U^{m}, U^{n})$, inside $U$. Moreover, $(G/B)(k)$ parametrizes complete flags $$U_1 \inj U_2 \inj \cdots \inj U_{m+n-1} \inj U_{m+n} = U.$$ Therefore, a point $(g_1B(k),g_2K(k))$ in $(G/B)(k) \times X(k)$ gives a complementary quiver representation $V = (U_i, U^m, U^n)$ of $Q_{m+n+2}$, where $U_i = g_1 \langle e_1, e_2, \dots, e_i \rangle$, $U^m = g_2 \langle e_1, e_2, \dots, e_m \rangle$, and $U^n = g_2 \langle e_{m+1}, e_{m+2}, \dots, e_{m+n} \rangle$. It is clear that this assignment sends points in a $G(k)$-orbit to isomorphic quiver representations. Conversely, every complementary quiver representation produces a complete flag and a pair of complementary subspaces. This completes the proof.     
\end{proof}

\begin{thm}
    There is a bijection between the set of $B(k)$-orbits in $X(k)$ and the set of $m$-matchings in the corona graph $C_{m+n}$.
\end{thm}

\begin{proof}
Combining Lemma \ref{l: comp rep sum admissible} and Proposition \ref{p: bijection quiver and orbit}, it suffices to show that the set of expressions of $d_{m,n}$ as sums of distinct admissible roots is in bijection with the set of $m$-matchings in the corona graph $C_{m+n}$.

We now describe the desired bijection. We label the vertices of $C_{m+n}$ as in Figure \ref{fig: C6C4}, i.e. the vertices of the complete graph are $v_i$, and each $v_i$ has an extra edge connecting to another vertex $w_i$. 

Let $d_{m,n} = \sum_{\alpha \in S} \alpha$ be an expression. For $1 \le i \le m+n$, there is exactly one $\alpha \in S$ which contains the term $L_i$. The desired bijection is constructed as follows:

    \begin{enumerate}
        \item if $\alpha = L_i+L_j \in S$ for $1 \le i \neq j \le m+n$, then we color the edge connecting the node $v_i$ and $v_j$;
        \item if $\alpha = L_i - L_{m+n+2} \in S$, then we color the edge connecting the node $v_i$ and $w_i$;
        \item if $\alpha = L_i + L_{m+n+2} \in S$, then we color no edges.
    \end{enumerate}

It is straightforward to see this assignment is a bijection between the expressions of $d_{m,n}$ as a sum of distinct admissible roots and the $m$-matchings in the corona graph $C_{m+n}$ (see Example \ref{exmp: matching assignments} below for a concrete illustration of this assignment). 
\end{proof}

\begin{defn}
    Given a $m$-matching $\cS$ in $C_{m+n}$, we define $\cO_{\cS}$ to be the corresponding $B(k)$-orbit in $X(k)$. Conversely, given a $B(k)$-orbit $\cO$ in $X(k)$, we define $\cS_{\cO}$ to be the corresponding $m$-matching in $C_{m+n}$.
\end{defn}

It is easy to construct a $k$-point in $\cO_{\cS}$. In fact we can choose the point to be represented by a binary matrix in $GL_{m+n}(k)$, i.e. every entry is either $0$ or $1$. 

\begin{prop}\label{p: binary representative}
    Using the identification of $G$ with $GL_{m+n}$, there exists a binary matrix $g \in G(k)$ such that $gK(k)$ is in $\cO_{\cS}$. 
\end{prop}

\begin{proof}
    Let $d_{m,n} = \sum_{\alpha \in S} \alpha$ be the  expression associated to $\cS$. 
    
    We fix a disjoint union decomposition of the set $\{1, 2, \cdots, m+n\} = A_m \sqcup A_n$, where $A_i$ is of cardinality $i$. Let $\chi$ be a permutation of $\{1, 2, \cdots, m+n\}$ satisfying the following conditions: 
    \begin{enumerate}
        \item if $L_i+L_j \in S$, then $\chi(i) \in A_m$ and $\chi(j) \in A_n$;
        \item if $L_i - L_{m+n+2} \in S$, then $\chi(i) \in A_m$;
        \item if $L_i + L_{m+n+2} \in S$, then $\chi(i) \in A_n$.
    \end{enumerate}
    It is easy to see such $\chi$ exists. Let $g$ be the unique element in $G(k)$ satisfying the following conditions: 
    \begin{enumerate}
        \item if $L_i+L_j \in S$ and $i < j$, then $g(e_{\chi(i)}) = e_i + e_j$ and  $g(e_{\chi(j)}) = e_j$;
        \item if $L_i - L_{m+n+2} \in S$, then $g(e_{\chi(i)}) = e_i$;
        \item if $L_i + L_{m+n+2} \in S$, then $g(e_{\chi(i)}) = e_i$.
    \end{enumerate}
    By construction, the quiver representation $(U_i,U^m,U^n)$ associated to $(B,gK)$ is given by $$
    U_i = \langle e_1, e_2, \cdots e_i \rangle, \quad U^m = g \langle e_j : j \in A_m \rangle, \quad U^n = g \langle e_j : j \in A_n \rangle.
    $$ 
    The above conditions force such a representation has the decomposition type $d_{m,n} = \sum_{\alpha \in S} \alpha$. Therefore, we have constructed the desired $g \in G(k)$.
\end{proof}

Let $K'$ be $GL_n \times GL_m$ and $X' = G/ K'$. 
By interchanging the branching arrows of the quiver $Q_p$, we know that there is a natural bijection $\DD$ between $B(k)$-orbits in $X(k)$ and $B(k)$-orbits in $X'(k)$. Translating this bijection to the matchings settings, we know that it corresponds to the following operations. Suppose $\cS$ is an $m$-matching in $C_{m+n}$, then the $n$-matching $\DD(\cS)$ can be described as follows:
\begin{enumerate}
    \item if the edge connecting $v_i$ and $v_j$ is chosen in $\cS$, then it is also chosen in $\DD(\cS)$;
    \item if the edge connecting $v_i$ and $w_i$ is chosen in $\cS$, then it is \emph{not} chosen in $\DD(\cS)$;
    \item if the edge connecting $v_i$ and $w_i$ is not chosen in $\cS$, then it is chosen in $\DD(\cS)$.
\end{enumerate}

\begin{exmp}\label{exmp: matching assignments}
    The $2$-matching $\cS$ of $C_6$ in Figure \ref{fig: S and dual S} corresponds to $$d_{2,4} = (L_2 + L_4) + (L_1-L_8) + (L_3+L_8) + (L_5+L_8) +(L_6 +L_8).$$ Meanwhile, its dual $\DD \cS$ is the $4$-matching corresponding to $$d_{4,2} = (L_2 + L_4) + (L_1+L_8) + (L_3-L_8) + (L_5-L_8) +(L_6 - L_8).$$ 
\end{exmp}

\begin{figure}[h]
\centering

\begin{minipage}[b]{0.45\textwidth}
\centering

\begin{tikzpicture}
[every node/.style={circle, draw, inner sep=1pt}, scale=0.5]

\foreach \i [count=\j from 0] in {1,...,6} {
    \node (v\i) at ({120 + 360/6 * \j}:2) {$v_\i$};
}

\foreach \i in {1,...,6} {
    \node (w\i) [circle, draw, inner sep=0.4pt] at ({120 + 360/6 * (\i-1)}:4) {$w_\i$};
    \draw[line width=0.5mm] (v1) -- (w1); 
    \draw (v\i) -- (w\i); 
}

\foreach \i in {1,...,6} {
    \foreach \j in {1,...,6} {
        \ifnum\i<\j
            \ifnum\i=2 \ifnum\j=4
                \draw[line width=0.5mm] (v\i) -- (v\j); 
            \else
                \draw (v\i) -- (v\j); 
            \fi
            \else
                \draw (v\i) -- (v\j); 
            \fi
        \fi
    }
}

\end{tikzpicture}

\end{minipage}
\hfill
\begin{minipage}[b]{0.45\textwidth}
\centering
\begin{tikzpicture}
[every node/.style={circle, draw, inner sep=1pt}, scale=0.5]

\foreach \i [count=\j from 0] in {1,...,6} {
    \node (v\i) at ({120 + 360/6 * \j}:2) {$v_\i$};
}

\foreach \i in {1,...,6} {
    \node (w\i) [circle, draw, inner sep=0.4pt] at ({120 + 360/6 * (\i-1)}:4) {$w_\i$};
    \draw[line width=0.5mm] (v3) -- (w3); 
    \draw[line width=0.5mm] (v5) -- (w5);
    \draw[line width=0.5mm] (v6) -- (w6);
    \draw (v\i) -- (w\i); 
}

\foreach \i in {1,...,6} {
    \foreach \j in {1,...,6} {
        \ifnum\i<\j
            \ifnum\i=2 \ifnum\j=4
                \draw[line width=0.5mm] (v\i) -- (v\j); 
            \else
                \draw (v\i) -- (v\j); 
            \fi
            \else
                \draw (v\i) -- (v\j); 
            \fi
        \fi
    }
}

\end{tikzpicture}
\end{minipage}

\caption{A $2$-matching $\cS$ and its dual $\DD \cS$ in $C_6$}
\label{fig: S and dual S}
\end{figure}

\begin{rem}
    The results in this section can be generalized to Bruhat decomposition for $GL_n$ using the following quiver. 

\begin{figure}[H]
\begin{center}  
\begin{tikzpicture}[>=Stealth, scale=0.8, transform shape]

\node (1) at (0, 0) {1};
\node (2) at (2, 0) {2};
\node (dots) at (4, 0) {$\cdots$};
\node (mid-1) at (6, 0) {$n-1$};
\node (mid) at (8, 0) {$n$};
\node (p) at (10, 0) {$n-1$};
\node (dots2) at (12, 0) {$\cdots$};
\node (q) at (14, 0) {$2$};
\node (r) at (16, 0) {$1$};

\draw[->] (1) -- (2);
\draw[->] (2) -- (dots);
\draw[->] (dots) -- (mid-1);
\draw[->] (mid-1) -- (mid);
\draw[<-] (mid) -- (p);
\draw[<-] (p) -- (dots2);
\draw[<-] (dots2) -- (q);
\draw[<-] (q) -- (r);

\end{tikzpicture}
\end{center}
\caption{Quiver for Bruhat decomposition for $GL_n$}
\end{figure}

\end{rem}

\section{Generalizations to other classical symmetric pairs} \label{s: other classical}

In this section, we extend the combinatorial framework established for type AIII to other classical symmetric pairs arising from involutions. Throughout this discussion, we assume $\operatorname{char}(k) \neq 2$, ensuring that the general theory for symmetric varieties developed in Section \ref{ss: general theory} remains applicable. Specifically, we demonstrate how matchings in modified corona graphs continue to parametrize Borel orbits for the symplectic and orthogonal cases. 

\subsection{General theory for symmetric varieties} \label{ss: general theory}
Let $G$ be a reductive group split over $k$, and let $\theta$ be an involution of $G$ defined over $k$. We denote by $K = G^{\theta}$ the subgroup of fixed points, which is a symmetric subgroup of $G$. As before, let $X$ denote the symmetric variety $G/K$. Throughout, we fix a $\theta$-invariant split maximal torus $T$ of $G$, whose existence is guaranteed by \cite[Proposition 2.3]{HW93}, and let $N$ denote its normalizer. Furthermore, we assume the existence of a $\theta$-stable Borel $k$-subgroup $B$ of $G$ containing $T$, which we fix for the remainder of this section.

The following fundamental result regarding the structure of symmetric varieties was established in \cite[Lemma 4.1]{S} under the assumption that $k$ is algebraically closed. However, the argument provided therein does not rely on the algebraic closure of the base field and carries over directly to an arbitrary field $k$. For completeness, we provide the full proof below.

Let $\mathcal{I}(k)$ denote the set of $\theta$-twisted involutions in $G(k)$, defined by:
\[ \mathcal{I}(k) = \{x \in G(k) : \theta(x) = x^{-1}\}. \]
The group $G(k)$ acts on $\mathcal{I}(k)$ via $\theta$-twisted conjugation, where $g \cdot x = g x \theta(g)^{-1}$. The symmetrization map $\tau: G(k)/K(k) \to \mathcal{I}(k)$, which sends the coset $gK(k)$ to the twisted involution $g \theta(g)^{-1}$, is an equivariant embedding with respect to this action.

\begin{prop} \label{p: two representatives}
    Any $B(k)$-orbit in $\mathcal{I}(k)$ meets $N(k)$. 
\end{prop}

\begin{proof}
    Let $U$ be the unipotent radical of $B$ and $U^-$ be the unipotent radical of the opposite Borel. For any $w \in W$, let $U_w = U \cap w U w^{-1}$ and $U_w' = U \cap w U^{-} w^{-1}$. Let $x \in \mathcal{I}(k)$. By the Bruhat decomposition over $k$, we write $x = u n u'$ uniquely, where $n \in N(k)$, $u \in U(k)$, and $u' \in U_w'(k)$ with $w$ being the image of $n$ in the Weyl group. The condition $\theta(x) = x^{-1}$ implies $\theta(u) \theta(n) \theta(u') = (u')^{-1} n^{-1} u^{-1}$. Because the $\theta$-stability of $B$ and $T$ implies that $\theta$ maps $N(k)$, $U(k)$, and $U^-(k)$ to themselves, the uniqueness of the Bruhat decomposition yields $\theta(n) = n^{-1}$.
    
    Setting $(\theta(u))^{-1} = u_1 u_1'$ with $u_1 \in U_w(k)$ and $u_1' \in U_w'(k)$, the uniqueness of the Bruhat decomposition further implies $u' = u_1'$ and $n^{-1} \theta(u_1) n = u_1^{-1}$. Define a $k$-automorphism $\psi$ of $U_w$ by $\psi(g) = n^{-1} \theta(g) n$. Note that $\psi$ is an involution and $\psi(u_1) = u_1^{-1}$.

    By Lemma \ref{l: cohomology trivial} below, we know that there exists a $u_2 \in U_w(k)$ such that $u_1 = \psi(u_2)u_2^{-1}$. It follows that $n u_1 = \theta(u_2) n u_2^{-1}$, showing that the $B(k)$-orbit of $x$ contains $n \in N(k)$.
\end{proof}

\begin{lem} \label{l: cohomology trivial}
    Let $\psi$ be a $k$-automorphism of order $2$ of a $k$-split unipotent group $U$. Then $\operatorname{H}^{1}(\psi, U(k))$ is trivial.
\end{lem}

\begin{proof}
    We proceed by induction on the dimension of $U$. 
    
    If $\dim U = 1$, then $U$ is isomorphic to the additive group $\mathbb{G}_a$. The $1$-cocycle condition for $u \in U(k)$ is $u + \psi(u) = 0$ in additive notation. As $\chara(k) \neq 2$, we may take $v = u/2$, and observe that $u = v \psi(v)^{-1}$ is solvable in $k$. Thus $\on{H}^1(\psi, U(k)) = \{1\}$.

    Suppose $\dim U > 1$. There exists a $\psi$-invariant central subgroup $Z \subset U$ defined over $k$ such that $Z \cong \mathbb{G}_a$. The exact sequence $1 \to Z \to U \to U/Z \to 1$ yields the following sequence in cohomology:
    \[ \on{H}^1(\psi, Z(k)) \to \on{H}^1(\psi, U(k)) \to \on{H}^1(\psi, (U/Z)(k)). \]
    By the inductive hypothesis, $\on{H}^1(\psi, (U/Z)(k))$ is trivial, and by the base case, $\on{H}^1(\psi, Z(k))$ is trivial. It follows that $\on{H}^1(\psi, U(k))$ is trivial.
\end{proof}

The following corollary is a direct consequence of Proposition \ref{p: two representatives}.
\begin{cor} \label{c: normalizer representative}
    For every $B(k) \times K(k)$-orbit $\mathcal{O}$ in $G(k)$, there exists a representative $g \in G(k)$ such that the corresponding twisted involution $\tau(g) = g\theta(g)^{-1}$ lies in the normalizer $N(k)$ of the $\theta$-stable maximal torus $T$.
\end{cor}

The following lemma is a direct application of the uniqueness of Bruhat decomposition. 

\begin{lem} \label{l: normalizer uniqueness}
    Let $n_1, n_2 \in N(k) \cap \mathcal{I}(k)$. Then $n_1$ and $n_2$ lie in the same $B(k)$-orbit under the $\theta$-twisted action if and only if they lie in the same $T(k)$-orbit. That is, $n_1 = b n_2 \theta(b)^{-1}$ for some $b \in B(k)$ if and only if $n_1 = t n_2 \theta(t)^{-1}$ for some $t \in T(k)$.
\end{lem}

\subsection{The matching correspondence}

In Section 4, we utilized the theory of quiver representations to identify $B(k)$-orbits on $GL_{m+n}(k)/GL_m(k) \times GL_n(k)$ with $m$-matchings in the corona graph $C_{m+n}$. We now show that this combinatorial description is not merely a tool for indexing, but is the direct manifestation of the group-theoretic data described in Section \ref{ss: general theory}.

Throughout this subsection, let $G = GL_{m+n}$ and $K = GL_m \times GL_n$. Recall from Corollary~\ref{c: normalizer representative} that for any $B(k)$-orbit in $G(k)/K(k)$, there exists a representative $g \in G(k)$ such that the twisted involution $\tau(g) = g\theta(g)^{-1}$ lies in $N(k) \cap \mathcal{I}(k)$. The image of $\tau(g)$ in the Weyl group $W \cong S_{m+n}$ is a $\theta$-twisted involution. By Lemma~\ref{l: normalizer uniqueness}, this image is independent of the choice of representative $g$ within the orbit. In the specific case of type AIII, the involution $\theta$ acts trivially on $W$; consequently, the image of $\tau(g)$ is a standard involution in the symmetric group $S_{m+n}$. We now demonstrate that the matchings derived from quiver representations recover the Weyl group involutions.

Let $U$ be a $k$-vector space of dimension $m+n$. Let $U^{m}$ be the subspace spanned by $e_i$ for $i \in \{1, \dots, m\}$ and $U^{n}$ be the subspace spanned by $e_i$ for $i \in \{ m+1, \dots, m+n\}$.  Let $\phi$ be the automorphism of $U$ defined by:
$$
    \phi(u) =
    \begin{cases}
    u & \text{if } u \in U^{m}, \\
    -u & \text{if } u \in U^{n}.
    \end{cases}
$$

\begin{thm} \label{t: matching weyl correspondence}
    Let $\mathcal{O}$ be a $B(k)$-orbit in $G(k)/K(k)$, and let $w \in S_{m+n}$ be the involution associated to $\mathcal{O}$ via the image of the representative $\tau(g) = g\theta(g)^{-1} \in N(k)$ in the Weyl group. Then $w$ is precisely the permutation defined by the matching $S_{\mathcal{O}}$ in the corona graph $C_{m+n}$ as follows:
    \begin{enumerate}
        \item There exists an edge between vertices $v_i$ and $v_j$ in the matching $S_{\mathcal{O}}$ if and only if $w(i) = j$ (with $i \neq j$).
        \item A vertex $v_i$ is unmatched or matched with the vertex $w_i$ in $S_{\mathcal{O}}$ if and only if $w(i) = i$.
    \end{enumerate}
\end{thm}

\begin{proof}
The involution $w$ is constructed from the linear action of the twisted involution $\tau(g) = g\theta(g)^{-1}$ on the standard basis $\{e_1, \dots, e_{m+n}\}$. Let $u_i = g^{-1}(e_i)$ for each $i \in \{1, \dots, m+n\}$. Since $\tau(g) \in N(k)$, its action on the basis is given by $\tau(g)(e_i) = c_i e_{w(i)}$ for some $c_i \in k^\times$. This is equivalent to the condition:
\begin{equation}\label{eq: theta_relation_cases}
    \phi(u_i) =
    \begin{cases}
    c_i u_{w(i)} & \text{if } i \le m, \\
    -c_i u_{w(i)} & \text{if } i > m.
    \end{cases}
\end{equation}

Recalling the decomposition $U = U^m \oplus U^n$ into the $+1$ and $-1$ eigenspaces of $\phi$, we write each basis vector as $u_i = x_i + y_i$, where $x_i \in U^m$ and $y_i \in U^n$. We have $\phi(u_i) = x_i - y_i$. 

The lines spanned by the vectors $u_i, x_i,$ and $y_i$ constitute the internal data of an indecomposable sub-representation of the quiver $Q_{m+n}$. The combinatorial structure of the matching $S_{\mathcal{O}}$ is then determined by the support of these vectors:
\begin{itemize}
    \item If $y_i = 0$, then $u_i \in U^m$. Condition \eqref{eq: theta_relation_cases} forces $w(i) = i$. This corresponds to the vertex $v_i$ being matched with $w_i$ in $S_{\mathcal{O}}$.
    \item If $x_i = 0$, then $u_i \in U^n$. Condition \eqref{eq: theta_relation_cases} forces $w(i) = i$. In this case, $v_i$ remains unmatched in $S_{\mathcal{O}}$.
    \item If both $x_i \neq 0$ and $y_i \neq 0$, then $u_i$ must be paired with another basis vector $u_j$ to satisfy \eqref{eq: theta_relation_cases}. This implies $w(i) = j$ for $j \neq i$, and the resulting pairing is represented by an edge between $v_i$ and $v_j$ in the matching $S_{\mathcal{O}}$.
\end{itemize}
\end{proof}

\subsection{Symplectic case (type CII)} \label{ss: symplectic}
In this section, we apply the group-theoretic framework developed in the previous subsections to the case of the symplectic group. Our goal is to demonstrate that the orbits of a Borel subgroup on a symmetric variety of type CII admit a combinatorial parametrization via matchings, analogous to the general linear case. To achieve this, we realize the symplectic group as the fixed-point subgroup of an involution on $GL(U)$ and define the symmetric variety structure through a second, commuting involution.

In this subsection, let $U$ be a $k$-vector space of dimension $2m+2n$ and $G = GL(U)$. Suppose we have a (split) non-degenerate skew-symmetric bilinear form $\langle \cdot , \cdot \rangle$ on $U$, i.e. there exists a basis $e_j$ for $j \in [-m-n,-1] \cup [1,m+n]$ such that $\langle e_{-j}, e_{j} \rangle = 1$ if $j>0$ and $\langle e_{k}, e_{j} \rangle = 0$ if $k+j \neq 0$. 

We fix such a basis. Consequently, we obtain a Lagrangian subspace $U_{-1}$ of $U$, i.e. the subspace spanned by $\{e_{-m-n}, e_{-m-n+1}, \cdots, e_{-1}\}$. More generally, let $U_{i}$ be the subspace of $U$ generated by $e_j$ for $j \le i$. Then we have a complete flag 
\begin{equation} \label{eq: complete flag}
    0 \inj U_{-m-n} \inj U_{-m-n+1} \inj \cdots \inj U_{-1} \inj U_{1} \inj \cdots \inj U_{m+n-1} \inj U_{m+n}=U.
\end{equation}
In other words, the natural ordering of $[-m-n,-1] \cup [1,m+n]$ produces a complete flag of $U$, and thus a Borel subgroup $B$ of $G$. 

The bilinear form induces an involution $\varphi$ on $GL(U)$ defined by the relation:
\begin{equation} \label{eq: varphi definition}
    \langle v, w \rangle = \langle g(v), \varphi(g)(w) \rangle, \quad v,w \in U.
\end{equation}

Concretely, if we identify $G$ with $GL_{2m+2n}$ via the basis $e_i$, then $\varphi(g) = Jg^{t,-1}J^{-1}$, where $J$ is the anti-diagonal matrix whose top right (resp.\ bottom left) quadrant entries are $-1$ (resp.\ $1$). Under this identification, $B$ becomes the subgroup of upper triangular matrices. Note that the fixed-point subgroup $G^{\varphi}$ is isomorphic to the split group $Sp_{2m+2n}$. Since $B$ is $\varphi$-stable, $B^{\varphi}$ is a Borel subgroup of $G^{\varphi}$.  

Let $U^{2m}$ be the subspace spanned by $e_i$ for $i \in [-m,-1] \cup [1,m]$ and $U^{2n}$ be the subspace spanned by $e_i$ for $i \in [-m-n,-m-1] \cup [m+1,m+n]$. The symplectic form on $U$ restricts to non-degenerate split symplectic forms on both $U^{2m}$ and $U^{2n}$. Let $\phi$ be the automorphism of $U$ such that 
$$
    \phi(u) =
\begin{cases}
u & \text{if } u \in U^{2m}, \\
-u & \text{if } u \in U^{2n}.
\end{cases}
$$

Let $\theta$ be the conjugation action of $\phi$ on $G$. Note that $\theta$ is an involution. Let $K$ be the fixed point subgroup of $\theta$ in $G$. Then $K$ is isomorphic to $GL(U^{2m}) \times GL(U^{2n})$. Since $\theta$ and $\varphi$ commute with each other, we know that $K$ is $\varphi$-stable and $K^{\varphi}$ is isomorphic to $Sp(U^{2m}) \times Sp(U^{2n})$. 

Similar to Lemma \ref{l: rational point G K}, we have the following result regarding the pair $(G^{\varphi}, K^{\varphi})$.

\begin{lem}\label{l: gk phi rational}
    The natural map $G^{\varphi}(k)/K^{\varphi}(k) \ra (G^{\varphi}/K^{\varphi})(k)$ is bijective.
\end{lem}

\begin{proof}
    Similar to Lemma \ref{l: rational point G K}, it suffices to show that $\on{H}^1(k, Sp_{2m})$ is trivial. Theorem 2.3.3 in \cite{GS17} states that the pointed set $\on{H}^1(k, Sp_{2m})$ classifies the $k$-isomorphism classes of non-degenerate alternating bilinear forms of rank $2m$ that become isomorphic to the standard form over $\ov{k}$. Since every such form over $k$ is equivalent to the standard alternating form, we have $\on{H}^1(k, Sp_{2m}) = 1$.
\end{proof}

Lemma \ref{l: gk phi rational} implies that the set of $B^{\varphi}(k)$-orbits in $(G^{\varphi}/K^{\varphi})(k)$ is in bijection with the set of double cosets $B^{\varphi}(k) \backslash G^{\varphi}(k) / K^{\varphi}(k)$.

\begin{rem}
    It is important to note that the map $G(k)/K(k) \ra (G/K)(k)$ may not be surjective for general symmetric pairs $(G,K)$. One example is $G = \gm$, $K = \mu_{2}$, and $k$ is a finite field.  
\end{rem}

Our goal is to study $B^{\varphi}(k) \bsl G^{\varphi}(k) / K^{\varphi}(k)$ via results about $B(k) \bsl G(k) / K(k)$ in Section \ref{s: GLmn}. Since $B(k)$ and $K(k)$ are $\varphi$-stable, there is a natural action of $\varphi$ on $B(k) \bsl G(k) / K(k)$. Under the combinatorial correspondence established previously, this is equivalent to an action of $\varphi$ on the set of $2m$-matchings in the corona graph $C_{2m+2n}$. We now proceed to describe this action explicitly.

We adopt a labeling for $C_{2m+2n}$ that differs slightly from the previous sections to better reflect the symplectic symmetry. The vertices of the complete subgraph are labeled $v_{i}$ for $i \in [-m-n,-1] \cup [1,m+n]$, and the corresponding extended vertices are denoted $w_{i}$. There is a graph automorphism of $C_{2m+2n}$ defined by the swaps $v_{i} \leftrightarrow v_{-i}$ and $w_{i} \leftrightarrow w_{-i}$ for all $i$. This automorphism induces an involution on the set of $2m$-matchings of $C_{2m+2n}$, which we denote by $\mathcal{S} \mapsto -\mathcal{S}$ and refer to as the \textit{minus involution}. See Figure \ref{fig: S and minus S} for an example of this action.

\begin{figure}[h]
\centering

\begin{minipage}[b]{0.45\textwidth}
\centering

\begin{tikzpicture}[scale=0.6,
  every node/.style={circle, draw, minimum size=7mm, inner sep=0pt}]

\node (v1) at (120:2) {$v_1$};
\node (v2) at (180:2) {$v_2$};
\node (v3) at (240:2) {$v_3$};
\node (v-3) at (300:2) {$v_{-3}$};
\node (v-2) at (0:2)   {$v_{-2}$};
\node (v-1) at (60:2)  {$v_{-1}$};

\node (w1) at (120:4) {$w_1$};
\node (w2) at (180:4) {$w_2$};
\node (w3) at (240:4) {$w_3$};
\node (w-3) at (300:4) {$w_{-3}$};
\node (w-2) at (0:4)   {$w_{-2}$};
\node (w-1) at (60:4)  {$w_{-1}$};

\draw[line width=0.6mm] (v1) -- (w1);
\draw (v2) -- (w2);
\draw (v3) -- (w3);
\draw (v-3) -- (w-3);
\draw (v-2) -- (w-2);
\draw (v-1) -- (w-1);

\foreach \i/\j in {
  v1/v2, v1/v3, v1/v-3, v1/v-2, v1/v-1,
  v2/v3, v2/v-3, v2/v-2, v2/v-1,
  v3/v-3, v3/v-2, v3/v-1,
  v-3/v-2, v-3/v-1,
  v-2/v-1}
  {
    \draw (\i) -- (\j);
  }

\draw[line width=0.6mm] (v2) -- (v-3);

\end{tikzpicture}
\end{minipage}
\hfill
\begin{minipage}[b]{0.45\textwidth}
\centering
\begin{tikzpicture}[scale=0.6,
  every node/.style={circle, draw, minimum size=7mm, inner sep=0pt}]

\node (v1) at (120:2) {$v_1$};
\node (v2) at (180:2) {$v_2$};
\node (v3) at (240:2) {$v_3$};
\node (v-3) at (300:2) {$v_{-3}$};
\node (v-2) at (0:2)   {$v_{-2}$};
\node (v-1) at (60:2)  {$v_{-1}$};

\node (w1) at (120:4) {$w_1$};
\node (w2) at (180:4) {$w_2$};
\node (w3) at (240:4) {$w_3$};
\node (w-3) at (300:4) {$w_{-3}$};
\node (w-2) at (0:4)   {$w_{-2}$};
\node (w-1) at (60:4)  {$w_{-1}$};

\draw (v1) -- (w1);
\draw (v2) -- (w2);
\draw (v3) -- (w3);
\draw (v-3) -- (w-3);
\draw (v-2) -- (w-2);
\draw [line width=0.6mm] (v-1) -- (w-1);

\foreach \i/\j in {
  v1/v2, v1/v3, v1/v-3, v1/v-2, v1/v-1,
  v2/v3, v2/v-3, v2/v-2, v2/v-1,
  v3/v-3, v3/v-2, v3/v-1,
  v-3/v-2, v-3/v-1,
  v-2/v-1}
  {
    \draw (\i) -- (\j);
  }

\draw[line width=0.6mm] (v-2) -- (v3);

\end{tikzpicture}
\end{minipage}

\caption{A $2$-matching $\cS$ and its minus $-\cS$ in $C_6$}
\label{fig: S and minus S}
\end{figure}

\begin{prop} \label{p: minus action}
    The $\varphi$-action on the $2m$-matchings in $C_{2m+2n}$ coincides with the minus involution, i.e. $\varphi(\cS) = -\cS$. 
\end{prop}

\begin{proof}
    It suffices to show that $\varphi(g)K(k)$ is in $\cO_{-\cS}$ for a particular $gK(k) \in \cO_{\cS}$. We verify the statement for the representative constructed in Proposition \ref{p: binary representative}. Let us fix the disjoint union decomposition of $[-m-n,-1]\cup [1,m+n] = A_{2m} \sqcup A_{2n}$, where $$A_{2m} = [-m,-1] \cup [1,m] \text{ and } A_{2n} = [-m-n, -m-1] \cup [m+1, m+n].$$ Note that both $A_{2m}$ and $A_{2n}$ are closed under negation. Moreover, we fix a suitable permutation $\chi$ of $[-m-n,-1]\cup [1,m+n]$ as in Proposition \ref{p: binary representative}. Then there exists a unique element $g$ in $G(k)$ satisfying the following conditions: 
    \begin{enumerate}
        \item if $v_i$ and $v_j$ are linked in $\cS$ and $i < j$, then $g(e_{\chi(i)}) = e_i + e_j$ and  $g(e_{\chi(j)}) = e_j$;
        \item if $v_i$ and $w_i$ are linked in $\cS$, then $g(e_{\chi(i)}) = e_i$;
        \item if $v_i$ and $w_i$ are not linked in $\cS$, then $g(e_{\chi(i)}) = e_i$.
    \end{enumerate}
    Using \eqref{eq: varphi definition}, it is easy to see that, for some suitable choices of the signs $\pm$, the element $\varphi(g) \in G(k)$ satisfies the following conditions: 
    \begin{enumerate}
        \item if $v_i$ and $v_j$ are linked in $\cS$ and $i< j$, then $\varphi(g)(e_{-\chi(i)}) = \pm e_{-i}$ and  $\varphi(g)(e_{-\chi(j)}) = \pm e_{-i} \pm e_{-j}$;
        \item if $v_i$ and $w_i$ are linked in $\cS$, then $\varphi(g)(e_{-\chi(i)}) = \pm e_{-i}$;
        \item if $v_i$ and $w_i$ are not linked in $\cS$, then $\varphi(g)(e_{-\chi(i)}) = \pm e_{-i}$.
    \end{enumerate} 
    Hence, $\varphi(g)K$ is in $\cO_{-\cS}$, i.e. $\varphi(\cS) = -\cS$.
    \end{proof}

Inspired by Proposition \ref{p: minus action}, we have the following definition.

\begin{defn}
    We say an edge in $C_{2m+2n}$ is \emph{horizontal} if it is fixed by the minus involution, i.e. it connects $v_{i}$ and $v_{-i}$ for some $i$.
\end{defn}

Next, we determine which $\varphi$-invariant double coset 
of $B(k) \bsl G(k) / K(k)$ contains a fixed point of $\varphi$. 

\begin{lem} \label{l: no horizontal edge}
    Let $g$ be an element in $G^{\varphi}(k)$ and $\cO$ be the double coset $B(k) g K(k)$. The matching $\cS$ corresponding to $\cO$ does not contain any horizontal edges.
\end{lem}

\begin{proof}
    Let $V$ be the quiver representation of $Q_{2m+2n+2}$ corresponding to $gK(k)$, that is the quiver representation given by the standard complete flag in \eqref{eq: complete flag} and the two subspaces $gU^{2m}$ and $gU^{2n}$. Since $g \in G^{\varphi}$, we know that $gU^{2m}$ and $gU^{2n}$ are orthogonal to each other. 

    We prove the statement by contradiction. Suppose the matching $\cS$ contains the edge connecting $v_{-i}$ and $v_{i}$. Let $V'$ be the indecomposable summand of $V$ corresponding to the root $L_{-i}-L_{i}$ for $i > 0$. Let $W$ be the $2$-dimensional subspace of $U$ given by $V'$. Then $W$ is spanned by the vector $e_{-i}$ and a vector $u \in U_{i} \bsl U_{i-1}$. Note that the pairing $\langle u, e_{-i} \rangle$ is non-zero. Therefore, $W$ is non-degenerate with respect to the symplectic form. 

    Since $W$ is indecomposable as quiver representation, there exist constants $a,b \in k$ such that $u + ae_{-i}$ (resp.\ $u + be_{-i}$) is in $W \cap gU^{2m}$ (resp.\ $W \cap gU^{2n}$). Since $U^{2m}$ and $U^{2n}$ are disjoint, we know that $a \neq b$. Moreover, since $U^{2m}$ and $U^{2n}$ are orthogonal and $\langle u, e_{-i} \rangle$ is non-zero, we know that $a=b$. Contradiction. Therefore, $\cS$ does not contain any horizontal edges.
\end{proof}

We now relate these results to the Weyl groups of $G$ and $G^{\varphi}$. Since $B^{\varphi}$ is a Borel subgroup of $G^{\varphi}$, the involution $\varphi$ acts on the Weyl group $W$ of $G$. The fixed-point subgroup $W^{\varphi}$ under this action is precisely the Weyl group of $G^{\varphi}$. Concretely, if we identify $W$ with the permutation group on the set $\{-m-n, \dots, -1, 1, \dots, m+n\}$, then $W^{\varphi}$ consists of the elements in $W$ that commute with the involution $i \leftrightarrow -i$.

From the symmetric variety perspective, let $\tau(g) = g \theta(g)^{-1}$ be a representative in the normalizer and consider its image in $W$. Following the notation in the proof of Theorem \ref{t: matching weyl correspondence}, we let $u_i = g^{-1}(e_i)$ for $i \in [-m-n, m+n] \setminus \{0\}$. Lemma \ref{l: no horizontal edge} ensures that the involution $w \in W$ corresponding to a symplectic orbit never maps $i$ to $-i$. 

Geometrically, this stems from the fact that $\phi(u_i)$ cannot be a multiple of $u_{-i}$; since $\phi$ is an involution, the identity $\langle u_i, \phi(u_i) \rangle = \langle \phi(u_i), \phi^2(u_i) \rangle = \langle \phi(u_i), u_i \rangle = -\langle u_i, \phi(u_i) \rangle$ would force $\langle u_i, \phi(u_i) \rangle = 0$, contradicting the non-degeneracy of the pairing between $e_i$ and $e_{-i}$.

Lemma \ref{l: no horizontal edge} provides a necessary condition for a double coset to contain $\varphi$-fixed point. Let us show that this condition is also sufficient.

\begin{prop}\label{p: no horizontal}
    Let $\cO$ be a double coset in $B(k) \bsl G(k) / K(k)$ such that $\cS_{\cO} = -\cS_{\cO}$. If $\cS_{\cO}$ does not contain any horizontal edges, then it contains a $k$-point of $G^{\varphi}$.
\end{prop}

\begin{proof}
    To construct such a $k$-point, it suffices to construct the corresponding quiver representation. By the assumption of having no horizontal edges, we write $\cS_{\cO} = \sqcup (\alpha \cup \varphi\alpha)$ as the disjoint union of a pair of edges. The statement can be proved in a similar way as in Proposition \ref{p: binary representative} and we will omit the proof. Basically, it reduces to showing that the desired quiver representation exists for $Sp_4$ and a pair of edges without horizontal edges in $C_4$. Moreover, one can check that the entries of $g$ can be chosen in the set $\{0,1,-1,1/2,-1/2\} \subset k$. 
\end{proof}

We observe that $\mathcal{O}^{\varphi}$, the set of $\varphi$-fixed points in $\mathcal{O}$, is stable under the action of $B^{\varphi}(k) \times K^{\varphi}(k)$. A natural question is to determine the number of double cosets in $B^{\varphi}(k) \backslash G^{\varphi}(k) / K^{\varphi}(k)$ contained in $\mathcal{O}^{\varphi}$.

When $k$ is algebraically closed, Wyser showed that $\mathcal{O}^{\varphi}$ consists of a single double coset (cf.\ \cite[Theorem 1.5.8]{WyPhD}). The original proof involves an involved combinatorial counting of clans with extra structure. In the following, we provide a simpler proof using the results from Section \ref{ss: general theory}. Notably, our approach demonstrates that even for a general field $k$, $\mathcal{O}^{\varphi}$ is indeed a single double coset, showing that this structural property is independent of the algebraic closure of the base field.

The following elementary lemma concerns the relationship between the tori $T^{\varphi}(k)$ and $T(k)$. If we identify $T$ with the diagonal matrices, then the map $\varphi: T \ra T $ is given by 
$$
    \varphi: (a_{-m-n}, \cdots, a_{-1}, a_{1}, \cdots, a_{m+n}) \mapsto (a_{m+n}^{-1}, \cdots, a_{1}^{-1}, a_{-1}^{-1}, \cdots, a_{-m-n}^{-1}).
$$
 
The following lemma establishes the $k$-rationality of the torus reduction, which is central to our main result. 

\begin{lem} \label{l: reduction_rational}
    Let $w \in W^{\varphi}$ be an involution and let $t \in T(k)$ satisfy $tw(t)^{-1} \in T^{\varphi}(k)$. If $w(i) \neq -i$ for all indices $i$, then there exists $t' \in T^{\varphi}(k)$ such that 
    $$t'w(t')^{-1} = tw(t)^{-1}.$$
\end{lem}

\begin{proof}
    Any involution $w \in W^\varphi$ without horizontal swaps can be decomposed into a product of disjoint involutions that, for $i \neq \pm j$, act by the transpositions $(i, j)$ and $(-i, -j)$ while fixing all other indices $q \notin \{\pm i, \pm j\}$. It suffices to show that the statement holds for such an involution.
    
    In this case, the condition $tw(t)^{-1} \in T^{\varphi}$ implies that the $i$-th and $(-i)$-th coordinates are inverses:
    $$
        a_i a_j^{-1} = (a_{-i} a_{-j}^{-1})^{-1} \implies a_i a_{-i} = a_j a_{-j}.
    $$
    We define the coordinates of $t'$ as follows:
    $$
        a'_i := a_i / a_j, \quad a'_j := 1, \quad a'_{-i} := a_j / a_i, \quad a'_{-j} := 1,
    $$
    and for $q \notin \{\pm i, \pm j\}$, we set $a'_q$ to be any value (such as $1$) such that $a'_q a'_{-q} = 1$. It is clear that $t' \in T^\varphi(k)$. One can readily verify that the coordinates of $t'w(t')^{-1}$ match those of $tw(t)^{-1}$. Since these definitions only involve ratios and inverses of the coordinates of $t \in T(k)$, the resulting element $t'$ is rational over $k$ and satisfies the required equality.
\end{proof}

\begin{rem} \label{r: rationality_obstacle}
    When $k$ is not algebraically closed, the condition $w(i) \neq -i$ is necessary. If $w$ were to swap $i$ and $-i$, the relation $a'_i/a'_{-i} = a_i/a_{-i}$ coupled with $a'_i a'_{-i} = 1$ would force $(a'_i)^2 = a_i/a_{-i}$, which generally requires a square root in $k$. 
\end{rem}

\begin{thm} \label{t: single coset}
    Let $\cO$ be a double coset as in Proposition \ref{p: no horizontal}. Then $\cO^{\varphi}$ is a single double coset in $B^{\varphi}(k) \bsl G^{\varphi}(k)/ K^{\varphi}(k)$.
\end{thm}

\begin{proof}
    We prove by contradiction. Suppose $\cO^{\varphi}$ contains two different double cosets. We choose elements $g_1$ and $g_2$ representing these distinct double cosets. By Proposition \ref{p: two representatives}, we may choose them such that$$\tau(g_i):= g_i \theta(g_i)^{-1} \in N_{G^{\varphi}}(T^{\varphi}).$$
    
    Since $g_1$ and $g_2$ belong to the same orbit $\cO$, Lemma \ref{l: normalizer uniqueness} implies that there exists a $t \in T(k)$ such that $t\tau(g_1) \theta(t)^{-1} = \tau(g_2)$, equivalently, $$tw_1(\theta (t)^{-1})\tau(g_1) = \tau(g_2),$$ where $w_1$ is the image of $\tau(g_1)$ in $W^{\varphi}$. Note that $\theta$ acts trivial on $T$ and $w_1$ is an involution without horizontal swap by Lemma \ref{l: no horizontal edge}. Lemma \ref{l: reduction_rational} implies that there exists a $t' \in T^{\varphi}(k)$ such that $$t'w_1(\theta (t')^{-1})\tau(g_1) = \tau(g_2).$$ From this and Lemma \ref{l: normalizer uniqueness}, we know that $g_1$ and $g_2$ belong to the same double coset in $B^{\varphi}(k) \bsl G^{\varphi}(k)/ K^{\varphi}(k)$. Contradiction.    
\end{proof}

Combining Lemma \ref{l: no horizontal edge}, Proposition \ref{p: no horizontal}, and Theorem \ref{t: single coset}, we know that the set of double cosets $B^{\varphi}(k) \bsl G^{\varphi}(k)/ K^{\varphi}(k)$ is in natural bijection with the set of minus-invariant $2m$-matchings in $C_{2m+2n}$ that contains no horizontal edges. Notice that the horizontal edges are the only minus-invariant edges in $C_{2m+2n}$. Moreover, the quotient of the graph $C_{2m+2n}$ with the horizontal edges removed, taken with respect to the minus action, is precisely, $C_{m+n}^{(2)}$. Therefore, we have proved the following theorem.

\begin{thm} \label{t: main thm Sp}
    There is a bijection between the set of $m$-matchings in the double corona graph $C_{m+n}^{(2)}$ and the set of double cosets $B^{\varphi}(k) \bsl G^{\varphi}(k)/ K^{\varphi}(k)$.
\end{thm}


\subsection{Orthogonal case (type BI)} \label{ss: orthogonal group}
The arguments developed for the symplectic group carry over to the case of orthogonal groups with minimal modifications. However, a key distinction arises regarding the rationality of the orbits: while the symplectic case holds over a general field $k$, the orthogonal case generally requires the algebraic closure $\bar{k}$. This requirement stems from the existence of horizontal edges in the matching; as discussed in Remark \ref{r: rationality_obstacle}, such edges necessitate the existence of square roots in the base field to represent a rational point in the orbit.

We adopt a labeling for $C_{2m+2n+1}$ where the nodes of the complete subgraph are $v_{i}$ for $i \in \{-m-n, \dots, m+n\}$, and the corresponding extended vertices are denoted $w_{i}$. Note that the index set now includes $0$, reflecting the odd dimension of the space. The graph admits an automorphism defined by $v_{i} \leftrightarrow v_{-i}$ and $w_{i} \leftrightarrow w_{-i}$, which induces the minus involution $\mathcal{S} \mapsto -\mathcal{S}$ on the set of $2m$-matchings.

\begin{thm}\label{t: orthogonal}
    Let $G$ be the split orthogonal group $O_{2m+2n+1}$, and let $K$ be the symmetric subgroup $O_{2m} \times O_{2n+1}$. Let $B$ be a Borel subgroup of $G$. There is a bijection between the set of minus-invariant $2m$-matchings in the corona graph $C_{2m+2n+1}$ and the set of double cosets $B(\ov{k}) \bsl G(\ov{k})/ K(\ov{k})$. 
\end{thm}

\begin{rem}
    We add some remarks on the theorem.
    \begin{itemize}
        \item By applying the duality $\DD$, we could replace with the set of minus-invariant $2n+1$-matchings in the corona graph $C_{2m+2n+1}$.
        \item Unlike Lemma $\ref{l: no horizontal edge}$ in the symplectic group case, there are no constraints on the double coset $\cO_{\cS}$ to support a point in $O_{2m+2n+1}(\ov{k})$ other than $\cS = -\cS$.
        \item It is clear that the statement in Theorem \ref{t: orthogonal} is equivalent to the corresponding statement for the special analogues of $G$, $K$, and $B$, that is, when each group is replaced by its subgroup of elements with determinant one.
        \item These arguments are expected to extend naturally to other classical symmetric pairs arising from involutions of $(GL_{m+n}, GL_m \times GL_n)$, including $(Sp_{2n}, GL_n)$, $(SO_{2n}, GL_n)$, and $(O_{2m+2n}, O_{2m} \times O_{2n})$.
    \end{itemize} 
\end{rem}

\begin{exmp} \label{ex: orthogonal_rationality_failure}
To illustrate that the bijection in Theorem \ref{t: orthogonal} does not necessarily descend to a general field $k$, consider the case $k = \mathbb{F}_3$. Let $U$ be a $5$-dimensional $k$-vector space with a basis $\{e_{-2}, e_{-1}, e_0, e_1, e_2\}$. We equip $U$ with a non-degenerate symmetric bilinear form $\langle \cdot, \cdot \rangle$ defined by $\langle e_i, e_j \rangle = \delta_{i, -j}$. In matrix form, this corresponds to the antidiagonal identity matrix $J$. 

Let $G = O_5$ be the orthogonal group preserving this form, consisting of matrices $g$ such that $g J g^T = J$. We consistently index the rows and columns of our matrices by the ordered set $\{-2, -1, 0, 1, 2\}$. Let $K = O_2 \times O_3$ be the symmetric subgroup defined by the involution $\theta(g) = D g D^{-1}$, where $D = \text{diag}(1, 2, 2, 2, 1)$.

Consider the elements $g_1, g_2 \in G(\mathbb{F}_3)$ given by:
\[
g_1 = \begin{pmatrix} 1 & 2 & 0 & 2 & 2 \\ 2 & 1 & 0 & 2 & 2 \\ 0 & 0 & 2 & 0 & 0 \\ 2 & 2 & 0 & 1 & 2 \\ 2 & 2 & 0 & 2 & 1 \end{pmatrix}, \quad 
g_2 = \begin{pmatrix} 2 & 2 & 0 & 1 & 2 \\ 1 & 1 & 0 & 1 & 2 \\ 0 & 0 & 2 & 0 & 0 \\ 2 & 1 & 0 & 1 & 1 \\ 2 & 1 & 0 & 2 & 2 \end{pmatrix}.
\]
Calculating the norm matrices $M_i = g_i \theta(g_i)^{-1}$, we find:
\[
M_1 = \begin{pmatrix} 0 & 0 & 0 & 0 & 2 \\ 0 & 0 & 0 & 2 & 0 \\ 0 & 0 & 1 & 0 & 0 \\ 0 & 2 & 0 & 0 & 0 \\ 2 & 0 & 0 & 0 & 0 \end{pmatrix}, \quad
M_2 = \begin{pmatrix} 0 & 0 & 0 & 0 & 1 \\ 0 & 0 & 0 & 1 & 0 \\ 0 & 0 & 1 & 0 & 0 \\ 0 & 1 & 0 & 0 & 0 \\ 1 & 0 & 0 & 0 & 0 \end{pmatrix}.
\]

Note that $M_1$ and $M_2$ are not conjugate under the action of the torus $T(\mathbb{F}_3)$, though they are conjugate over $T(\mathbb{F}_9)$. By Lemma \ref{l: normalizer uniqueness}, $g_1$ and $g_2$ represent distinct double cosets in $B(\mathbb{F}_3) \backslash G(\mathbb{F}_3) / K(\mathbb{F}_3)$ that identify over $\bar{k}$. This illustrates that the set of minus-invariant matchings only provides a lower bound for the number of rational orbits in the orthogonal case.
\end{exmp}

\section{Applications} \label{s: applications}
We now apply the results in Sections \ref{s: GLmn} and \ref{s: other classical} to derive numerical results on the number of double cosets.  

Recall the following definitions in the introduction. Let $a_{m+n,m}$ (resp.\ $c_{m+n,m}$) be the number of orbits in Theorem \ref{t: main thm} (1) (resp.\ (2)). We recover the recurrence relation obtained in \cite{CU1, CU}.

\begin{prop} \label{p: recurrences}
    The sequences $a_{p,m}$ and $c_{p,m}$ satisfy the recurrences:
    $$a_{p,m} = a_{p-1,m-1} + a_{p-1,m} + (p-1) a_{p-2,m-1},$$
    $$c_{p,m} = c_{p-1,m-1} + c_{p-1,m} + 2(p-1) c_{p-2,m-1}.$$
\end{prop}

\begin{proof}
    Using the matching description for $a_{p,m}$ from Theorem \ref{t: main thm}, we proceed by induction on $p$, the number of vertices in the complete subgraph of $C_{p}$. 
    If we choose the edge connecting $v_1$ and $w_1$, there are $a_{p-1,m-1}$ matchings. If we choose the edge connecting $v_1$ and $v_i$, there are $a_{p-2,m-1}$ matchings for each $i \neq 1$. If we color no edges of $v_1$, there are $a_{p-1,m}$ matchings. Hence, the proof is complete for $a_{p,m}$. The proof is similar for $c_{p,m}$.
\end{proof}

\begin{prop} \label{p: ultra logconcave}
    Fix an integer $p>0$. The sequence $(a_{p,m})$ is symmetric unimodal for $0 \le m \le p$. Moreover, we have the inequality $$a_{p,m}^2 \ge \left(1+\frac{1}{m}\right) \left(1 + \frac{1}{p-m}\right)a_{p,m-1}a_{p,m+1},$$ for $0 < m < p$, that is $(a_{p,m})$ is ultra log-concave. The analogous statements hold for $c_{p,m}$.
\end{prop}

\begin{proof}
    It is clear that the sequence $(a_{p,m})$ is symmetric. The statements that unimodality and ultra log-concavity of $(a_{p,m})$ are general facts about matching polynomials of graphs \cite{HL}. The same arguments apply to $(c_{p,m})$.
\end{proof}

We provide the growth rate of $c_{p,q}$ as $p \rightarrow \infty$ in the following lemma.

\begin{lem} \label{l: estimation}
    For a fixed $q \in \N$, $c_{p,q} = O(\frac{1}{q!}p^{2q})$.
\end{lem}

\begin{proof}
    Theorem \ref{t: main thm Sp} implies that $c_{p,q}$ is the number of $q$-matchings in the double corona graph $C^{(2)}_p$. Note that there are $p^2$ edges in $C^{(2)}_p$. When $p$ is sufficiently large, almost all collection of $q$ edges forms a matching. Therefore, 
    $c_{p,q} = O(\binom{p^2}{q})$.
\end{proof}

Assume that $\ov{k}$ is algebraically closed. Let $b_{m,n}$ be the number of the Borel orbits of the symmetric variety $SO_{2m+2n+1}/S(O_{2m} \times O_{2n+1})$. Can and U\u{g}urlu \cite{CU} showed that $b_{m,n}$ is a polynomial in $m$ when $n$ is fixed. They conjectured that the polynomial has integral coefficients only when $n = 0$. We confirm their conjecture.

\begin{prop}
    The polynomial $b_{m,n}$ is of degree $2n+1$, whose leading coefficient is non-integral when $n > 0$.     
\end{prop}

\begin{proof}
    By Theorem \ref{t: orthogonal}, it suffices to show that the statement for the number of minus-invariant $2n+1$-matchings in the corona graph $C_{2m+2n+1}$. When $m$ is sufficiently large, almost all collections of $2n+1$ edges forms a matching. Therefore, it suffices to count the number of minus-invariant $2n+1$ edges subsets in $C_{2m+2n+1}$. 

    There are two types of minus-invariant $(2n+1)$-edges subsets in $C_{2m+2n+1}$, distinguished by whether they contain the edge $e_0$ connecting $v_0$ and $w_0$.

    Suppose the subset contains $e_0$ and exactly $2l$ horizontal edges, where $l \le n$. Then the number of such subsets is $\binom{m+n}{2l}c_{m+n-2l,n-l}$. Therefore, there are $$ O \left(\sum_{l=0}^{n}\binom{m+n}{2l}\frac{1}{(n-l)!}(m+n-2l)^{2(n-l)} \right) $$ subsets containing $e_0$ by Lemma \ref{l: estimation}. Note that the top degree of the expression is $m^{2n}$.  

    Suppose the subset contains exactly $2l+1$ horizontal edges and does not contain $e_0$, where $l \le n$. Then the number of such subsets is $\binom{m+n}{2l+1}c_{m+n-2l-1,n-l}$. Therefore, there are $$ O \left(\sum_{l=0}^{n}\binom{m+n}{2l+1}\frac{1}{(n-l)!}(m+n-2l-1)^{2(n-l)} \right) $$ subsets not containing $e_0$. Note that the top degree of the expression is $m^{2n+1}$.

    Therefore, we know that $$b_{m,n} = O\left(m^{2n+1}\sum_{l=0}^{n}\frac{1}{(2l+1)!(n-l)!}\right).$$ 

    One can verify that the sum $$\sum_{l=0}^{n}\frac{1}{(2l+1)!(n-l)!}$$ is strictly less than $1$ when $n \ge 2$. In fact, this sum decreases exponentially when $n$ goes to infinity. In particular, this implies that the leading coefficient of $b_{m,n}$ is not an integer.
\end{proof}

\bibliographystyle{amsplain}
\bibliography{reference}

\end{document}